\documentclass{amsart}
\usepackage{amssymb}
\usepackage{hyperref}
\usepackage{cite}
 \newtheorem{thm}{Theorem}[section]
 \newtheorem{cor}[thm]{Corollary}
 \newtheorem{lem}[thm]{Lemma}
 \newtheorem{prop}[thm]{Proposition}
 \theoremstyle{definition}
 
 \theoremstyle{remark}
 \newtheorem{rem}[thm]{Remark}
 \newtheorem{ex}[thm]{Example}
 \numberwithin{equation}{section}

\begin{document}

\title[ Common local spectral properties]
{Further common local spectral properties for bounded linear operators
}

\author[]{Hassane Zguitti}

\address{Hassane Zguitti: 
Department of Mathematics, Dhar El Mahraz Faculty of Science, Sidi Mohamed Ben Abdellah University, BO 1796 Fes-Atlas, 30003 Fez Morocco.
}

\email{hassane.zguitti@usmba.ac.ma}

\begin{abstract} In this note, we study common local spectral properties for bounded linear operators $A\in\mathcal{L}(X,Y)$ and $B,C\in\mathcal{L}(Y,X)$ such that $$A(BA)^2=ABACA=ACABA=(AC)^2A.$$ We prove that $AC$ and $BA$ share the single valued extension property, the Bishop property $(\beta)$, the property $(\beta_{\epsilon})$, the decomposition property $(\delta)$ and decomposability. Closedness of analytic core and quasinilpotent part are also investigated. Some applications to Fredholm operators are given.

\end{abstract}

%%% ----------------------------------------------------------------------
\maketitle
%%% ----------------------------------------------------------------------
%\tableofcontents

\section{Introduction}
For any Banach spaces $X$ and $Y$, let $\mathcal{L}(X,Y)$ denote the set of all bounded linear operators from $X$ to $Y$; with $\mathcal{L}(X)=\mathcal{L}(X,X)$. For $T\in\mathcal{L}(X)$, let $\mathcal{N}(T)$, $\mathcal{R}(T)$, $\sigma(T)$, $\sigma_p(T)$, $\sigma_a(T)$ and $\sigma_s(T)$ denote the null space, the range, the spectrum, the point spectrum, the approximate point spectrum and the surjective spectrum of $T$, respectively.  An operator $T\in\mathcal{L}(X)$ is said to be an {\it upper semi-Fredholm} operator if
 ${\mathcal R}(T)$ is closed and $\dim \mathcal{N}(T)<\infty$, and $T$ is said to be a {\it lower semi-Fredholm} operator if $\mbox{codim}(T)<\infty$. $T$ is said to be {\it Fredholm} operator if $\dim \mathcal{N}(T)<\infty$ and $\mbox{codim}(T)<\infty$. The {\it upper semi-Fredholm spectrum} $\sigma_{uf}(T)$, {\it lower semi-Fredholm spectrum} $\sigma_{lf}(T)$ and  the {\it essential spectrum} $\sigma_e(T)$ are defined by
 $$\sigma_{uf}(T)=\{\lambda\in\mathbb{C}\,:\,T-\lambda\mbox{ is not upper semi-Fredholm}\};$$
 $$\sigma_{lf}(T)=\{\lambda\in\mathbb{C}\,:\,T-\lambda\mbox{ is not lower semi-Fredholm}\};$$
 $$\sigma_{e}(T)=\{\lambda\in\mathbb{C}\,:\,T-\lambda\mbox{ is not Fredholm}\}.$$
For an arbitrary $T\in\mathcal{L}(X)$, the {\it local resolvent set} $\rho_T(x)$ of $T$ at a vector $x$ in $X$ is defined to consist of all $\lambda\in\mathbb{C}$ for which there exists an analytic $X$-valued function $f$ on an open neighborhood $U$ of $\lambda$ such that $$(T-\mu)f(\mu)=x,\,\,\mbox{ for all }\mu\in U.$$ The {\it local spectrum} $\sigma_T(x)$ is defined by $\sigma_T(x)=\mathbb{C}\setminus\rho_T(x)$. The local spectrum $\sigma_T(x)$ is a subset of $\sigma(T)$ and it may happen to be empty. Moreover, we have (see \cite{LN}) $$\sigma_s(T)={\displaystyle\bigcup_{x\in X}}\sigma_T(x).$$
 
 For  $T\in\mathcal{L}(X)$ and $F\subseteq \mathbb{C}$, let $X_T(F)$ denote the {\it local spectral subspace} defined by $$X_T(F)=\{x\in X\,:\,\sigma_T(x)\subseteq F\}.$$ Clearly, $X_T(F)$ is a linear (not necessarily closed) subspace of $X$. The operator $T$ is said to possess the {\it Dunford's property} $(\mathcal{C})$ if $X_T(F)$ is closed for every closed subset $F$ of $\mathbb{C}$.
 \smallskip
 
 The
 operator $T\in{\mathcal L}(X)$ is said to have the {\it single valued
 extension property} (SVEP, for short) at $\lambda\in\mathbb{C}$ provided that there
 exists an open disc $V$ centered at $\lambda$  such that for every open subset $U \subset V$, the constant function $f\equiv 0$ is the only
 analytic solution of the equation
 $$(T-\mu)f(\mu)=0\quad\forall\mu\in U.$$ We use ${\mathcal S}(T)$ to denote the  set where $T$ fails to have the SVEP  and we say that $T$ has the SVEP if
$ {\mathcal S}(T)$ is the empty set, \cite{Fi, LN}. In the case where $T$ has the SVEP, $\sigma_T(x)=\emptyset$ if and only if $x=0$. Moreover (\cite[Lemma 3]{LV}), $$\sigma(T)=\sigma_s(T)\cup\mathcal{S}(T).$$

 \indent For  an open set $U$ of $\mathbb{C}$, let $\mathcal{O}(U,X)$ be the Fr\'echet space of all $X$-valued analytic function on $U$ endowed with the topology defined by uniform convergence on every compact subset of $U$.  An operator $T\in\mathcal{L}(X)$ is said to satisfy the {\it Bishop's
 property ($\beta$) on an open set} $U\subseteq {C}$ provided that for every open subset $V$ of $U$ and for any
 sequence $(f_n)_n$ of analytic $X$-valued functions on $V$, 
 $$(T-\mu)f_n(\mu)\longrightarrow 0\mbox{ in } \mathcal{O}(V,X)\Longrightarrow f_n(\mu)\longrightarrow 0\mbox{ in } \mathcal{O}(V,X).$$ 
Let $\rho_\beta(T)$ be the largest open set on which $T$ has the property $(\beta)$. Its complement $\sigma_\beta(T)=\mathbb{C}\setminus\rho_\beta(T)$ is a closed, possibly empty, subset of $\sigma(T)$. Then   $T$ is said to satisfy the 
 Bishop's property ($\beta$), precisely when $\sigma_\beta(T) = \emptyset$, 
\cite{AE, MMN}.
 \smallskip
 
 It is well known that the following implications hold
 $$\mbox{Bishop's property }(\beta)\Rightarrow\mbox{ Dunford's property }(\mathcal{C})\Rightarrow\mbox{ SVEP}.$$
 
\indent In order to introduce the dual notion of Bishop's property $(\beta)$, we need a slight variant of the local spectral subspace. For a closed subset $F$ in $\mathbb{C}$, the  {\it glocal spectral analytic space} $\mathcal{X}_T(F)$ is the linear subspace of vectors $x\in X$ for which there exists an analytic function $f\,:\,\mathbb{C}\setminus F\rightarrow X$ such that $$(T-\mu)f(\mu)=x,\,\,\mbox{ for all }\mu\in \mathbb{C}\setminus F.$$
We point out that the analytic function $f$ is defined globally on the entire complement of $F$. Evidently, $ \mathcal{X}_T(F)$ is linear subspace contained in $X_T(F)$. Moreover, the equality $\mathcal{X}_T(F)=X_T(F)$ holds for all closed sets $F\subseteq \mathbb{C}$ precisely when $T$ has the SVEP \cite[Proposition 3.3.2]{LN}.
 \medskip
 
 \indent An operator $T\in\mathcal{L}(X)$ is said to have the {\it decomposition
 property $(\delta)$ on } $U$ provided that for all open sets $V,W\subseteq \mathbb{C}$ for which $\mathbb{C}\setminus U\subseteq V\subseteq\overline{V}\subseteq W$, we have
 \begin{equation}{\mathcal X}_T(\mathbb{C}\setminus V)+{\mathcal X}_T(\overline{ W})=X.
 \end{equation}
 Let $\rho_\delta(T)$ be the  largest open set on which the operator $T$ has the property $(\delta) $. Its complement $\sigma_\delta(T)=\mathbb{C}\setminus\rho_\delta(T)$ is a closed, possibly empty, subset of $\sigma(T)$ ( \cite[Corollary 17]{MMN}). Then   $T$  has the decomposition property ($\delta$) if $\sigma_{\delta}(T)=\emptyset$.
 \smallskip
 
 Properties ($\beta$) and ($\delta$) are known to be dual to
 each other in the sense that $T$ has $(\delta)$ on $U$ if and only if
 $T^*$ satisfies
 $(\beta)$ on $U$ \cite{AE, MMN}. Moreover $$ \sigma_\beta(T)=\sigma_\delta(T^*)\mbox{ and }\sigma_\delta(T)=\sigma_\beta(T^*).$$

 \indent The operator $T\in{\mathcal L}(X)$ is said to be {\it
 decomposable on } $U$ provided that for every finite open cover
$\{U_1,\ldots,U_n\}$ of
 $\mathbb{C},$ with $\sigma(T)\setminus U\subseteq U_1$, there exists
 $X_1,\ldots,X_n$ closed  $T$-invariant subspaces of $X$ for which 
 \begin{equation}\label{eq00}\sigma(T|X_i)\subseteq U_i\,\hbox{ for
 }\,i=1,\ldots,n\,\hbox{ and }\,X_1+\cdots+X_n=X.\end{equation}
 Let $\rho_{dec}(T)$ be the largest open set $U\subseteq\mathbb{C}$ on which $T$ is decomposable. Its  complement $\sigma_{dec}(T)=\mathbb{C}\setminus\rho_{dec}(T)$ is a closed, possibly empty, subset of $\sigma(T)$. We say that   $T$  is decomposable  if $\sigma_{dec}(T)=\emptyset$. The class of decomposable operators contains all normal operators and more
generally all spectral operators. Operators with totally disconnected spectrum
are decomposable by the Riesz functional calculus. In particular, compact
and algebraic operators are decomposable.
 
 It is also known that $(\beta)$ characterizes operators with
 decomposable extensions \cite{AE}.
   Property ($\beta $) is hence conserved by restrictions while
 $(\delta)$ is
 transferred to quotient operators. See also \cite{LN} for more
 details. We have $$\sigma_{dec}(T)=\sigma_{\beta}(T)\cup\sigma_{\delta}(T)=\sigma_{\beta}(T)\cup\sigma_{\beta}(T^*)=\sigma_{dec}(T^*).$$ 
 
 Let $\mathcal{E}(U,X)$ be the Fr\'echet space of all $X$-valued $C^\infty$-functions on $U$.  The operator $T$ is said to satisfy the {\it 
 property} $(\beta_\epsilon)$ at $\lambda\in\mathbb{C}$ provided that there exists open disc $U$ centered at $\lambda$  such
 that for every open subset $V\subset U$ and for any
 sequence $(f_n)_n$ of infinitely differentiable $X$-valued functions on $V$, we have
 $$(T-\mu)f_n(\mu)\longrightarrow 0\mbox{ in } \mathcal{E}(U,X)\Longrightarrow f_n(\mu)\longrightarrow 0\mbox{ in } \mathcal{E}(U,X).$$
Let $\sigma_{\beta_{\epsilon}}(T)$ be the set of all points where $T$ fails to satisfy the 
 property ($\beta_\epsilon$). Then   $T$ is said to satisfy the property ($\beta_\epsilon$), precisely when $\sigma_{\beta_{\epsilon}}(T) = \emptyset$. The property $(\beta_\epsilon)$ plays the same role for generalized scalar operators as Bishop's property $(\beta)$ does for decomposable operators: an operator $T$ satisfies $(\beta_\epsilon)$ if and only if $T$ is subscalar, is the sense that it has a generalized scalar extension.  An operator $T$ is said to be {\it generalized scalar} if there exists a continuous homomorphism algebra $\Phi\,:\,\mathcal{E}(\mathbb{C}\rightarrow\mathcal{L}(X)$ with $\Phi(1)=I$ and $\Phi(z)=T$ (see \cite{EP}).
 \smallskip
 
Jacobson's Lemma asserts that if $A\in\mathcal{L}(X,Y)$ and $B\in\mathcal{L}(Y,X)$ then
\begin{equation}\label{Jacob} AB-I\mbox{ is invertible if and only if } BA-I\mbox{ is invertible}.
\end{equation}
For $A\in\mathcal{L}(X,Y)$ and $B\in\mathcal{L}(Y,X)$, numerous mathematicians showed that $AB-I$ (resp. $AB$) and $BA-I$ (resp. $BA$) share many spectral properties, see \cite{Ai2, Ai3, Ba, BZ, CDH, LY, MZ, YF, ZZ1, ZZ2} and the references therein. For the local spectral properties, Benhida and Zerouali \cite{BZ} proved that $AB$ and $BA$ share the SVEP, Bishop property $(\beta)$, the property $(\beta_{\epsilon})$, the decomposition property $(\delta)$ and decomposability.  The Dunford condition $(\mathcal{C})$ was studied by Aiena and Gonzalez in \cite{Ai2,Ai3} for operators $A$ and $B$ such that $ABA=A^2$ and $BAB=B^2$. Then Zeng and Zhong \cite{ZZ2} extented common local spectral properties for $AC$ and $BA$ under the condition $ABA=ACA$. For operators $A$, $B$, $C$ and $D$ satisfying $ACD=DBD$ and $BDA=ACA$, Yan and Fang \cite{YF} investigated local spectral properties for $AC$ and $BD$. Recently, \cite{CS} studied the common properties for $ac$ and $ba$ for elements in a ring satisfying $a(ba)^2=abaca=acaba=(ac)^2a$.
 \smallskip
 
 In this note, we extend results of \cite{Ai2, Ai3, BZ, ZZ2} by studying common local spectral properties for bounded linear operators $A\in\mathcal{L}(X,Y)$ and $B,C\in\mathcal{L}(Y,X)$ such that $$A(BA)^2=ABACA=ACABA=(AC)^2A.$$ We prouve that $AC$ and $BA$ share the single valued extension property, the Bishop property $(\beta)$, the property $(\beta_{\epsilon})$, the decomposition property $(\delta)$ and decomposability. Closedness of analytic core and quasinilpotent part are also investigated. Some applications to Fredholm operators are given.

\section{common local spectral properties}

\begin{prop}\label{prop1} Let $A\in\mathcal{L}(X,Y)$ and $B,C\in\mathcal{L}(Y,X)$ such that $A(BA)^2=ABACA=ACABA=(AC)^2A$ and let $\lambda\in\mathbb{C}$. Then $AC$ has the SVEP at $\lambda$ if and only if $BA$ has the SVEP at $\lambda$.

In particular, $AC$ has the SVEP if and only if $BA$ has the SVEP.
\end{prop}

\begin{proof} Assume that $AC$ has the SVEP at $\lambda$ and let $f$ be an $X$-valued analytic function in a neighborhood $U$ of $\lambda$ such that
 \begin{equation}\label{eqsvep1}
(BA-\mu)f(\mu)=0,\,\forall \mu\in U.
\end{equation}
By taking $ABA$ values in equality (\ref{eqsvep1}) and using equality $A(BA)^2=ACABA$ , we obtain $(AC-\mu)ABAf(\mu)=0,\,\forall \mu\in U$. Since $ABAf(\mu)$ is analytic on $U$ and $AC$ has the SVEP at $\lambda$, then $ABAf(\mu)=0,\,\forall \mu\in U$. By taking $A$ values in equality (\ref{eqsvep1}), we get $\mu Af(\mu)=0,\,\forall \mu\in U$ and then $A f(\mu)=0,\,\forall \mu\in U$. Hence $\mu f(\mu)=0,\,\forall \mu\in U$. Thus $f(\mu)=0,\,\forall \mu\in U$. Therefore $BA$ has the SVEP at $\lambda$. 

Conversely, let $BA$ have the SVEP at $\lambda$. Then it follows from \cite[Proposition 2.1]{BZ} that $AB$ has the SVEP at $\lambda$. Now with the same argument as in the direct sense, we get that $CA$ has the SVEP at $\lambda$. Again by \cite[Proposition 2.1]{BZ},  $AC$ has the SVEP at $\lambda$.
\end{proof}
If $A\in\mathcal{L}(X,Y)$ and $B,C\in\mathcal{L}(Y,X)$ are such that $ABA=ACA$, then  the result of \cite[Theorem 9]{CDH} is an immediate consequence of Proposition \ref{prop1}.

\begin{thm}\label{thm1} Let $A\in\mathcal{L}(X,Y)$ and $B,C\in\mathcal{L}(Y,X)$ such that $A(BA)^2=ABACA=ACABA=(AC)^2A$. Then
$$\sigma_\beta(AC)=\sigma_\beta(BA).$$
In particular, $AC$ satisfies property $(\beta)$ if and only if $BA$ satisfies property $(\beta)$.
\end{thm}

\begin{proof} Assume that $AC$ satisfies the Bishop's property $(\beta)$ on some open set $U$ in $\mathbb{C}$.  Let $V$ be an open subset of $U$  and let $(f_n)_n$ be a sequence of $X$-valued analytic functions on  $V$ 
 such that
\begin{equation}\label{eq1.1}(BA-\mu)f_n(\mu)\longrightarrow 0\mbox{ in }\mathcal{O}(V,X).
 \end{equation}
 Then $$ACA(BA-\mu)f_n(\mu)=(AC-\mu )ACAf_n(\mu)\longrightarrow 0\mbox{ in }\mathcal{O}(V,X).$$
 Since $AC$ satisfies the Bishop's property $(\beta)$, then $$ACAf_n(\mu)\longrightarrow 0\mbox{ in } \mathcal{O}(V,X).$$ Hence $$ABACAf_n(\mu)=A(BA)^2f_n(\mu)\longrightarrow 0 \mbox{ in } \mathcal{O}(V,X).$$ Thus it follows from (\ref{eq1.1}) that $$\mu ABAf_n(\mu)\longrightarrow 0 \mbox{ in } \mathcal{O}(V,X).$$ So by \cite[Lemma 2.1]{BZ},   $ABAf_n(\mu)$ converges to zero in $ \mathcal{O}(V,X).$ Then by taking $A$ values in equality (\ref{eq1.1}) , $\mu Af_n(\mu)$ converges to zero in $ \mathcal{O}(V,X)$. Hence $Af_n(\mu)$ converges to zero in $ \mathcal{O}(V,X)$ by \cite[Lemma 2.1]{BZ}. Again by (\ref{eq1.1}),  $\mu f_n(\mu)$ and then $f_n(\mu)$ converges to zero in $ \mathcal{O}(V,X)$. Which prove that $BA$ satisfies the Bishop's property $(\beta)$ on $U$.
 
 Conversely, Assume that $BA$ satisfies the Bishop's property $(\beta)$ on $U$. Then it follows from the proof of  \cite[Proposition 2.1]{BZ} that $AB$ satisfies the Bishop's property $(\beta)$ on $U$. Hence by the same way we get that $CA$ satisfies the Bishop's property $(\beta)$ on $U$. Thus by \cite[Proposition 2.1]{BZ}, $AC$ satisfies the Bishop's property $(\beta)$ on $U$. 
\end{proof}

Since $(\beta)$ and $(\delta)$ are dual to each other, then we get from Theorem \ref{thm1}
\begin{thm}\label{thm1b} Let $A\in\mathcal{L}(X,Y)$ and $B,C\in\mathcal{L}(Y,X)$ such that $A(BA)^2=ABACA=ACABA=(AC)^2A$. Then
$$\sigma_\delta(AC)=\sigma_\delta(BA).$$
In particular, $AC$ satisfies property $(\delta)$ if and only if $BA$ satisfies property $(\delta)$.
\end{thm}

Since decomposability is equivalent to both $(\beta)$ and $(\delta)$, then it follows immediately from Theorem \ref{thm1} and Theorem (\ref{thm1b}):
\begin{cor}\label{C1} Let $A\in\mathcal{L}(X,Y)$ and $B,C\in\mathcal{L}(Y,X)$ such that $A(BA)^2=ABACA=ACABA=(AC)^2A$. Then
$$\sigma_{dec}(AC)=\sigma_{dec}(BA).$$
In particular, $AC$ is decomposable if and only if $BA$ is decomposable.
\end{cor}
In order to show that $AC$ and $BA$ share the property $(\beta_\epsilon)$, we need the following lemma
\begin{lem}{\rm \cite{BZ}}\label{lembz} Let $U$ be an open set and $(f_n)_n$ be a sequence in $\mathcal{E}(U,X)$ such that $(\mu f_n(\mu))_n$ converge to zero in $\mathcal{E}(U,X)$. Then $(f_n)_n$ converges to zero in $\mathcal{E}(U,X)$.
\end{lem}

\begin{thm}\label{thmeps} Let $A\in\mathcal{L}(X,Y)$ and $B,C\in\mathcal{L}(Y,X)$ such that $A(BA)^2=ABACA=ACABA=(AC)^2A$. Then 
$$\sigma_{\beta_\epsilon}(AC)=\sigma_{\beta_\epsilon}(BA).$$
In particular, $AC$ satisfies property $(\beta_\epsilon)$ if and only if $BA$ satisfies property $(\beta_\varepsilon)$.
\end{thm}
\begin{proof}
Now using Lemma \ref{lembz} and the same argument as in the proof of Theorem \ref{thm1} we get the result.
\end{proof}
\begin{cor} Let $A\in\mathcal{L}(X,Y)$ and $B,C\in\mathcal{L}(Y,X)$ such that $A(BA)^2=ABACA=ACABA=(AC)^2A$. Then $AC$ is subscalar if and only if $BA$ is subscalar.
\end{cor}
%--------------------------------------------------
\section{local spectrum and related subspaces}
%--------------------------------------------------
\begin{thm}\label{thm10} Let $A\in\mathcal{L}(X,Y)$ and $B,C\in\mathcal{L}(Y,X)$ such that $A(BA)^2=ABACA=ACABA=(AC)^2A$. Then \\
i) $\sigma_{AC}(ABAx)\subseteq\sigma_{BA}(x)\subseteq\sigma_{AC}(ABAx)\cup\{0\}$.\\
ii) $\sigma_{AB}(ACAx)\subseteq\sigma_{CA}(x)\subseteq\sigma_{AB}(ACAx)\cup\{0\}$.\\
iii) $\sigma_{BA}(BACACx)\subseteq\sigma_{AC}(x)\subseteq\sigma_{BA}(BACACx)\cup\{0\}$.
\end{thm}
\begin{proof} i) Let $\lambda\notin\sigma_{BA}(x)$. Then there exists an open neighborhood $U$ of $\lambda$ and an $X$-valued analytic function $f$ on $U$ such that $$(BA-\mu)f(\mu)=x\mbox{ for all }\mu\in U.$$ Hence for any $\mu\in U$, $ABAx=ABA(BA-\mu)f(\mu)=(ABABA-\mu ABA)f(\mu)=(AC-\mu)ABAf(\mu)$. Since $ABAf(\mu)$ is analytic on $U$, then $\lambda\notin\sigma_{AC}(ABAx)$. Thus $\sigma_{AC}(ABAx)\subseteq\sigma_{BA}(x)$.

Now let $\lambda\notin \sigma_{AC}(ABAx)\cup\{0\}$. By virtue of \cite[Proposition 3.1]{BZ}, $\lambda\notin \sigma_{CA}(BAx)\cup\{0\}$. Then there exists an open neighborhood $U$ of $\lambda$ with $0\notin U$ and an $X$-valued analytic function $f$ on $U$ such that $$(CA-\mu)f(\mu)=BAx\mbox{ for all }\mu\in U.$$ Hence $ABABAx=ABA(CA-\mu)f(\mu)=(AB-\mu) ABAf(\mu)$ for all $\mu\in U$. Thus $\lambda\notin \sigma_{AB}(ABABAx)\cup\{0\}$. Therefore
$$
\begin{array}{ccc}\lambda\notin \sigma_{AB}(ABABAx)\cup\{0\}&\Longrightarrow&\lambda\notin \sigma_{BA}(BABAx)\cup\{0\}\\
&\Longrightarrow&\lambda\notin\sigma_{AB}(ABAx)\cup\{0\}\\
&\Longrightarrow&\lambda\notin\sigma_{BA}(BAx)\cup\{0\}\\
&\Longrightarrow&\lambda\notin\sigma_{AB}(Ax)\cup\{0\}\\
&\Longrightarrow&\lambda\notin\sigma_{BA}(x)\cup\{0\}.
\end{array}
$$In the last implications we use repetitively (\cite[Proposition 3.1]{BZ}).\\
ii) goes similarly.\\
iii) follows from i) and ii).
\end{proof}

\begin{cor}\label{cor31} Let $A\in\mathcal{L}(X,Y)$ and $B,C\in\mathcal{L}(Y,X)$ such that $A(BA)^2=ABACA=ACABA=(AC)^2A$. The following assertions hold\\
\indent i) If $AC$ is injective, $\sigma_{AC}(ABAx)=\sigma_{BA}(x)$ for all $x\in X$.\\
\indent ii) If $AB$ is injective, $\sigma_{AB}(ACAx)=\sigma_{CA}(x)$ for all $x\in X$.
\end{cor}
\begin{proof} If $AC$ (resp. $AB$) is injective then $0\notin\sigma(AC)$ (resp. $0\notin\sigma(AC)$). Hence the result follows at once from  Theorem \ref{thm10}.
\end{proof}

In particular, if $A$, $B$ and $C$ are injective, then
$$\sigma_{AC}(ABAx)=\sigma_{BA}(x)\mbox{ and } \sigma_{AB}(ACAx)=\sigma_{CA}(x),\mbox{ for all }x\in X.$$

Before that we study the Dunford property $(\mathcal{C})$ for $AC$ and $BA$, we start by the following lemmas.
\begin{lem}\label{lem10} Let $A\in\mathcal{L}(X,Y)$ and $B,C\in\mathcal{L}(Y,X)$ such that $A(BA)^2=ABACA=ACABA=(AC)^2A$. Let $F$ be a closed subset of $\mathbb{C}$ such that $0\in F$. Then the following are equivalent:\\
\indent i) $Y_{AC}(F)$ is closed.\\
\indent ii) $X_{BA}(F)$ is closed.
\end{lem}
\begin{proof} Assume that $Y_{AC}(F)$ is closed and let $(x_n)_n$ be a sequence in $X_{BA}(F)$ which converges to some $x$ in $X$. Then $\sigma_{BA}(x_n)\subset F$. By Theorem \ref{thm10}, part i), we have $\sigma_{AC}(ABAx_n)\subset F$ and then $ABAx\in Y_{AC}(F)$. Since $ABAx_n$ converges to $ABAx$ and $Y_{AC}(F)$ is closed by assumption, then $ABAx\in Y_{AC}(F)$. Hence $\sigma_{AC}(ABAx)\subset F$. Since $0\in F$, then it follows from Theorem \ref{thm10}, part i), that $\sigma_{BA}(x)\subset F$. Therefore $x$ is in $X_{BA}(F)$ and $X_{BA}(F)$ is closed.

For the converse implication, if $X_{BA}(F)$ is closed then it follows from \cite[Lemma 2.4]{ZZ2} that $Y_{AB}(F)$ is closed. By the same above argument we prove that $X_{CA}(F)$ is closed. Hence by \cite[Lemma 2.4]{ZZ2}, $Y_{AC}(F)$ is closed
\end{proof}

\begin{lem}\label{lem11} Let $A\in\mathcal{L}(X,Y)$ and $B,C\in\mathcal{L}(Y,X)$ such that $A(BA)^2=ABACA=ACABA=(AC)^2A$ and $AC$ has the SVEP. Let $F$ be a closed subset of $\mathbb{C}$ such that $0\notin F$. \\
\indent i) If $Y_{AC}(F\cup\{0\})$ is closed then $X_{BA}(F)$ is closed.\\
\indent ii) If $X_{BA}(F\cup\{0\})$ is closed then $Y_{AC}(F)$ is closed.
\end{lem}
\begin{proof}
i) Assume that $Y_{AC}(F\cup\{0\})$ is closed. Since $0\in F\cup\{0\}$ then it follows from Lemma \ref{lem10} that $X_{BA}(F\cup\{0\})$ is closed. Since $BA$ has the SVEP by Theorem \ref{thm1}, then it follows from  \cite[Lemma 2.5]{ZZ2} that $X_{BA}(F)$ is closed.

ii) It follows similarly.

\end{proof}

\begin{thm} Let $A\in\mathcal{L}(X,Y)$ and $B,C\in\mathcal{L}(Y,X)$ such that $A(BA)^2=ABACA=ACABA=(AC)^2A$. Then $AC$ has the Dunford's property $(\mathcal{C})$ if and only if $BA$ has the Dunford's property $(\mathcal{C})$.
\end{thm}
\begin{proof} It follows at once from Lemma \ref{lem10} and Lemma \ref{lem11}.

\end{proof}

The analytical core of $T$ is the set $K(T)$ of all $x\in X$ such that there exist  a constant $c>0$ and a sequence $(x_n)_n\subset X$  such that $$x_0=x, Tx_n=x_{n-1}\mbox { and }\|x_n\|\leq c^n\|x\| \mbox{ for all }n\in\mathbb{N}.$$ Recall that (see for instance, \cite[Theorem 2.18]{Ai1} or \cite[Proposition 3.3.7]{LN}) that 
$$K(T-\lambda)=X_T(\mathbb{C}\setminus\{\lambda\})=\{x\in X\,:\,\lambda\notin\sigma_T(x)\}.$$
The analytic core was studied by Mbekhta \cite{Mb1,Mb2}. In general, $K(T)$ is not need to be closed. By virtue of \cite[Corollary 6]{MMN2}, for any non-invertible decomposable operator $T$, the point $0$ is isolated in $\sigma(T)$ exactly when $K(T)$ is closed. In particular, if $T$ is a compact operator, or more generally a Riesz operators, then $K(T)$ is closed precisely when $T$ has finite spectrum, \cite[Corollary 9]{MMN2}.

\begin{thm}\label{thmco1} Let $A\in\mathcal{L}(X,Y)$ and $B,C\in\mathcal{L}(Y,X)$ such that $A(BA)^2=ABACA=ACABA=(AC)^2A$. Then for all nonzero complex $\lambda$, $K(AC-\lambda)$ is closed if and only if $K(BA-\lambda)$ is closed.
\end{thm}
\begin{proof} Assume that $K(AC-\lambda)$ is closed and let $(x_n)_n$ be a sequence in $K(BA-\lambda)$ such that $(x_n)_n$ converges to $x\in X$. Since $(x_n)\subset K(BA-\lambda)$ then for all nonnegative integer $n$, $\sigma_{BA}(x_n)\subset\mathbb{C}\setminus\{\lambda\}$. Hence it follows from Theorem \ref{thm10} that $\sigma_{AC}(ABAx_n)\subset\mathbb{C}\setminus\{\lambda\}$. Thus the sequence $(ABAx_)$ belongs to $K(AC-\lambda)$. Since $ABAx_n\longrightarrow ABAx$ and $K(AC-\lambda)$ is closed, then $ABAx\in K(AC-\lambda)$ and so $\sigma_{AC}(ABAx)\subset\mathbb{C}\setminus\{0\}$. We deduce from Theorem \ref{thm10} that $\sigma_{BA}(x)\subset\mathbb{C}\setminus\{\lambda\}$, i.e, $x\in K(BA-\lambda)$. Therefore $K(BA-\lambda)$ is closed.

Since $K(BA-\lambda)$ is closed if and only if $K(AB-\lambda)$ is closed (\cite[Corollary 3.3]{ZZ2}), then with the same argument we can prove the reverse implication.
\end{proof}

Associated with $T\in\mathcal{L}(X)$ there is another linear subspace of $X$, the {\it quasinilpotent part} $H_0(T)$ of $T$ defined as $$H_0(T)=\{x\in X\,:\,\displaystyle{\lim_{n\rightarrow\infty}}\|T^nx\|^{1\over n}=0\}.$$In general, $H_0(T)$ is not need to be closed. By \cite[Proposition 3.3.13]{LN}, $$H_0(T)=\mathcal{X}_T(\{0\}).$$ 
\begin{thm}\label{thmq1}  Let $A\in\mathcal{L}(X,Y)$ and $B,C\in\mathcal{L}(Y,X)$ such that $A(BA)^2=ABACA=ACABA=(AC)^2A$. Then $H_0(AC)$ is closed if and only if $H_0(BA)$ is closed.
\end{thm}
\begin{proof} Assume that $H_0(AC)$ is closed  and let $(x_n)$ be a sequence in $H_0(BA)$ which converges to $x\in X$. It is easy to see that $ABAx_n$ belongs to $H_0(AC)$. Since $H_0(AC)$ is closed and $ABAx_n$ converges to $ABAx$, then $ABAx\in H_0(AC)$. From the equality  
$$\|(AC)^pABAx\|^{1\over p}=\|(AB)^{p+1}Ax\|^{1\over p}=(\|(AB)^{p+1}Ax\|^{1\over p+1})^{p\over p+1}$$
we have $Ax\in H_0(AB)$. Since
\[\begin{array}{ccc}\|(BA)^{p+1}x\|^{1\over p+1}&=&\|B(AB)^{p}Ax\|^{1\over p+1}\\
&\leq &\|B\|^{1\over p+1}(\|(AB)^{p}Ax\|^{1\over p})^{p\over p+1}\\
&\leq&M^{1\over p+1}(\|(AB)^{p}Ax\|^{1\over p})^{p\over p+1}\quad\mbox{ where } M=\max(\|B\|,1);
\end{array}\] 
then $x\in H_0(BA)$ and so $H_0(BA)$ is closed.

The reverse implication goes similarly.\end{proof}
%------------------------------------------
\section{Applications and concluding remarks}
%------------------------------------------

A weaker version of property $(\delta)$ can be given as follows (see  \cite{ZZg} for a localized version): an operator  $T$ is said to have the weak spectral property $(\delta_w)$ provided that for every finite open
 cover $\{U_1,\ldots,U_n\}$ of $\mathbb{C},$  we have
 \begin{equation}{\mathcal X}_T(\overline{U_1})+\cdots+{\mathcal X}_T(\overline{ U_n})\mbox{ is dense in }X.
 \end{equation}

\begin{prop}\label{weak} Let $A\in\mathcal{L}(X,Y)$ and $B,C\in\mathcal{L}(Y,X)$ such that $A(BA)^2=ABACA=ACABA=(AC)^2A$. Assume that $ABA$ and $ACA$ have dense ranges.\\
i) If $BA$ has the  property $(\delta_w)$ then $AC$ has the  property $(\delta_w)$.\\
ii)  If $CA$ has the  property $(\delta_w)$ then $AB$ has the  property $(\delta_w)$.
\end{prop}
\begin{proof} We only prove i). Assume that $BA$ has the weak property $(\delta_w)$. Let  $\{U_1,\ldots,U_n\}$ be a finite open cover of $\mathbb{C}.$  Then
 $${\mathcal X}_{BA}(\overline{U_1})+\cdots+{\mathcal X}_{BA}(\overline{ U_n})\mbox{ is dense in }X.
 $$
 It is easy to see that $ABA({\mathcal X}_{BA}(\overline{U_i}))\subseteq {\mathcal Y}_{AC}(\overline{U_i})$ for each $i$, $i=1,\ldots,n$. Since $ABA(X)$ is dense in $Y$, then it follows that $${\mathcal Y}_{AC}(\overline{U_1})+\cdots+{\mathcal Y}_{AC}(\overline{ U_n})\mbox{ is dense in }Y.$$
 Thus $AC$ has the property $(\delta_w)$.
\end{proof}
We do not know if we can drop the condition that $ABA$ and $ACA$ have dense ranges in the last proposition.

\begin{prop}\label{p1.3} Let $A\in\mathcal{L}(X,Y)$ and $B,C\in\mathcal{L}(Y,X)$ such that $A(BA)^2=ABACA=ACABA=(AC)^2A$. Then
$$AC-I\mbox{ is surjective if and only if }BA-I\mbox{ is surjective.}$$
In other word, $$\sigma_s(AC)\cup\{0\}=\sigma_s(BA)\cup\{0\}.$$
\end{prop}
\begin{proof}  Since for every bounded linear operator $T$ the surjective spectrum satisfies $\sigma_s(T)=\displaystyle{\bigcup_{x\in X}}\sigma_T(x)$, then it follows from Theorem \ref{thm10}, part i), that $\sigma_{s}(AC)\subseteq\sigma_{s}(BA)\cup\{0\}$. Also by Theorem \ref{thm10}, part iii), we get $\sigma_{s}(BA)\subseteq\sigma_{s}(AC)\cup\{0\}$. Thus, 
$$\sigma_s(AC)\cup\{0\}=\sigma_s(BA)\cup\{0\}.$$
\end{proof}

It is well known that for $T\in\mathcal{L}(X)$ we have $\sigma_a(T)=\sigma_s(T^*)$ and $\sigma_s(T)=\sigma_a(T^*)$ (see for instance, \cite[Proposition 1.3.1]{LN}). Then it follows from Proposition \ref{p1.3}:
\begin{prop}\label{p1.1} Let $A\in\mathcal{L}(X,Y)$ and $B,C\in\mathcal{L}(Y,X)$ such that $A(BA)^2=ABACA=ACABA=(AC)^2A$. Then
$$AC-I\mbox{ is bounded below if and only if }BA-I\mbox{ is bounded below.}$$
In other word, $$\sigma_a(AC)\setminus\{0\}=\sigma_a(BA)\setminus\{0\}.$$
\end{prop}
\begin{rem}A direct proof of Proposition \ref{p1.1} can be given as follows: Assume that $AC-I$ is bounded below. Then there exists $c>0$ such that $$\|x\|\leq c\|(AC-I)x\|,\,\,\,\forall x\in X.$$
Then for all $x\in X$, $$\begin{array}{ccc}\|ABAx\|&\leq&c\|(AC-I)ABAx\|\\
&=&c\|(ABABA-ABA)x\|\\
&\leq&c\|ABA\|\|(BA-I)x\|.
\end{array}$$
Hence $$\begin{array}{ccc}\|Ax\|&\leq&\|ABAx\|+\|ABAx-Ax\|\\
&\leq&c\|ABA\|\|(BA-I)x\|+\|A\|\|(BA-I)x\|\\
&\leq&c'\|(BA-I)x\| \,\,\,(c'=c\|ABA\|+\|A\|).
\end{array}$$
Thus
$$\begin{array}{ccc}\|x\|&\leq&\|BAx\|+\|(BA-I)x\|\\
&\leq&\|B\|\|Ax\|+\|(BA-I)x\|\\
&\leq&\|B\|c'\|(BA-I)x\|+\|(BA-I)x\|\\
&\leq&c''\|(BA-I)x\|\,\,\,(c''=\|B\|c'+1).
\end{array}$$
Therefore, $BA-I$ is bounded below. The other sense goes similarly.
\end{rem}

From Proposition \ref{p1.3} and Proposition \ref{p1.1} we retrieve the result of \cite[Lemme 3.1]{CS} for bounded linear operators:
\begin{cor}\label{c3.1} Let $A\in\mathcal{L}(X,Y)$ and $B,C\in\mathcal{L}(Y,X)$ such that $A(BA)^2=ABACA=ACABA=(AC)^2A$. Then
$$AC-I\mbox{ is invertible if and only if }BA-I\mbox{ is invertible.}$$
In other word, $$\sigma(AC)\setminus\{0\}=\sigma(BA)\setminus\{0\}.$$
\end{cor}

\begin{prop}\label{p1.2} Let $A\in\mathcal{L}(X,Y)$ and $B,C\in\mathcal{L}(Y,X)$ such that $A(BA)^2=ABACA=ACABA=(AC)^2A$. Then
$$AC-I\mbox{ is injective if and only if }BA-I\mbox{ is injective.}$$
In other word, $$\sigma_p(AC)\setminus\{0\}=\sigma_p(BA)\setminus\{0\}.$$
\end{prop}
\begin{proof}
Assume that $AC-I$ is injective. Let $x\in X$ such that $(BA-I)x=0$. Then $BAx=x$. Hence $ABABAx=ACABAx=ABAx$. It follows that $(AC-I)ABAx=0$. Since $AC-I$ is injective we deduce that $ABAx=0$. Then $Ax=0$. Thus $x=0$.

The other implication goes similarly.

\end{proof}

Let $l^\infty(X)$ be the set of all bounded sequences of elements of $X$. Endowed with the norm $\|(x_n)\|=\sup_n\|x_n\|$, $l^\infty(X)$ is a Banach space. For $\tilde{x}=(x_n)\in l^\infty(X)$ let $q(\tilde{x})$ be the infimum of all $\varepsilon>0$ such that the set $\{x_n\,:\,n\in\mathbb{N}\}$ is contained in the union of a finite number of open balls with radius $\varepsilon$. Let $$m(X)=\{\tilde{x}\in l^\infty(X)\,:\,q(\tilde{x})=0\}.$$ For $T\in{\mathcal L}(X)$ let $T^\infty$ be the bounded linear defined on $l^\infty(X)$ by $T^\infty((x_n))=(Tx_n).$ Set $\tilde{X}=l^\infty(X)/m(X)$ and let $\tilde{T}\,:\tilde{X}\rightarrow\tilde{X}$ be the operator defined by $\tilde{T}(\tilde{x}+m(X))=T^\infty\tilde{x}+m(X)$. Then by \cite[Theorem 17.6 and Theorem 17.9]{Muller} we have
$$T\mbox{ is upper semi-Fredholm }\Longleftrightarrow \tilde{T}\mbox{ is injective }\Longleftrightarrow \tilde{T}\mbox{ is bounded bellow}$$ and
$$T\mbox{ is lower semi-Fredholm }\Longleftrightarrow \tilde{T}\mbox{ is surjective }.$$

Now let $A\in\mathcal{L}(X,Y)$ and $B,C\in\mathcal{L}(Y,X)$ such that $A(BA)^2=ABACA=ACABA=(AC)^2A$.  Then 
$$\tilde{A}(\tilde{B}\tilde{A})^2=\tilde{A}\tilde{B}\tilde{A}\tilde{C}\tilde{A}=\tilde{A}\tilde{C}\tilde{A}\tilde{B}\tilde{A}=(\tilde{A}\tilde{C})^2\tilde{A}.$$

As an immediate consequence of Proposition \ref{p1.3} and Proposition \ref{p1.1} we get the following result.

\begin{prop} Let $A\in\mathcal{L}(X,Y)$ and $B,C\in\mathcal{L}(Y,X)$ such that $A(BA)^2=ABACA=ACABA=(AC)^2A$. Then\\
i) $AC-I\mbox{ is upper semi-Fredholm if and only if }BA-I\mbox{ is upper semi-Fredholm;}$\\
ii) $AC-I\mbox{ is lower semi-Fredholm if and only if }BA-I\mbox{ is lower semi-Fredholm}.$\\
In other word, $$\sigma_{uf}(AC)\setminus\{0\}=\sigma_{uf}(BA)\setminus\{0\};$$
$$\sigma_{lf}(AC)\setminus\{0\}=\sigma_{lf}(BA)\setminus\{0\}.$$
\end{prop}

\begin{cor} Let $A\in\mathcal{L}(X,Y)$ and $B,C\in\mathcal{L}(Y,X)$ such that $A(BA)^2=ABACA=ACABA=(AC)^2A$. Then
$$AC-I\mbox{ is Fredholm if and only if }BA-I\mbox{ is Fredholm.}$$
In other word, $$\sigma_e(AC)\setminus\{0\}=\sigma_e(BA)\setminus\{0\}.$$
\end{cor}

\begin{ex}\label{Ex1} Let $P$ be a non trivial idempotent on $X$. Let $A$, $B$ and $C$ defined on $X\oplus X\oplus X$ by
\[ 
A = \left(\begin{array}{ccc} 0 & I & 0\\
                                0 & P & 0\\
                                0 & 0 & 0
                                                  \end{array}  \right)\mbox{ et }
B = \left(\begin{array}{ccc} I & 0 & 0\\
                                0 & I & 0\\
                                0 & 0 & 0
                                                  \end{array}  \right)  \mbox{ and }
 C = \left(\begin{array}{ccc} 0 & 0 & 0\\
                                I & 0 & 0\\
                                0 & I & 0
                                                  \end{array}  \right).
\]                                
Then $A(BA)^2=ABACA=ACABA=(AC)^2A$ and $ABA \neq ACA$. Hence common spectral properties for $AC$ and $BA$ can only followed directly from the above results, but not from the corresponding ones in \cite{ZZ2}.
\end{ex}

\begin{ex} Let $A$ and $B$ be as in Example \ref{Ex1}  and let $C$ be defined on $X\oplus X\oplus X$ by
\[ 
 C = \left(\begin{array}{ccc} 0 & 0 & 0\\
                               P & 0 & 0\\
                                0 & I & 0
                                                  \end{array}  \right).
\]                                
Then $A(BA)^2=ABACA=ACABA=(AC)^2A$ and $ABA \neq ACA$. Thus common spectral properties for $AC$ and $BA$ can only followed directly from the above results, but not from the corresponding ones in \cite{ZZ2}.
\end{ex}

%%%%%%%%%%%%%%%%%%%%%%%%%%%%%%%%

%%%%%%%%%%%%%%%%%%%%%%%%
\end{document}